\documentclass{amsart}

\usepackage{amssymb,amsmath,amsthm,latexsym,amscd,amsfonts,mathrsfs,mathtools}
\usepackage{epsfig}
\usepackage{psfrag}
\usepackage{graphicx}
\usepackage{bbold}
\usepackage{color}

\newtheorem{Theorem}{\bf Theorem}
\newtheorem{lemma}[Theorem]{\bf Lemma}

\newtheorem{definition}[Theorem]{\bf Definition}

\newtheorem{remark}[Theorem]{\bf Remark}
\newtheorem{theorem}[Theorem]{\bf Theorem}

\newtheorem{notation}[Theorem]{\bf Notation}
\DeclareMathOperator*{\dprime}{\prime \prime}

\def\ss{\subseteq}

\def\qed{\hfill$\Box$}
\def\scfig #1 #2 {\resizebox{#2}{!}{\includegraphics{#1}}}

\newcommand{\be}{\begin{equation}}
\newcommand{\ee}{\end{equation}}

\numberwithin{equation}{section}

\begin{document}

\title[Commuting squares and planar subalgebras]{Commuting squares and planar subalgebras}

\author{Keshab Chandra Bakshi}
\address{Chennai Mathematical Institute, Chennai, India}
\author{Vijay Kodiyalam}
\address{The Institute of Mathematical Sciences, Chennai, India and Homi Bhabha National Institute, Mumbai, India}
\email{bakshi209@gmail.com,vijay@imsc.res.in}

\begin{abstract}
We show a close relationship between non-degenerate smooth commuting squares 
of $II_1$-factors with all inclusions of finite index and inclusions of subfactor planar
algebras by showing that each leads to a construction of the other.
One direction of this uses the Guionnet-Jones-Shlyakhtenko construction.
\end{abstract}

\date{June 14, 2019.}

%
%
%
%
%
%
%
\maketitle

\section{Introduction}

The goal of this paper is to establish a relationship between $*$-planar subalgebras of a planar algebra
and commuting squares. While results of this nature have been known for a while in the language of
standard $\lambda$-lattices - see \cite{Ppa1995} and \cite{Ppa2002} - we employ the Guionnet-Jones-Shlayakhtenko (GJS) construction - see \cite{GnnJnsShl2010} and \cite{JnsShlWlk2010} - to simplify the proofs considerably. Needless to say, the proofs are very
pictorial.

We begin with a review of the GJS construction in the version described in \cite{KdySnd2009} in \S2.
In \S 3 we start with a non-degenerate commuting square of $II_1$-factors
$$
\begin{matrix}
L &\subseteq & M\cr
\rotatebox{90}{$\subseteq$} &\ &\rotatebox{90}{$\subseteq$}\cr
N &\subseteq & K
\end{matrix}
$$ 
with all inclusions extremal of finite index and 
that is smooth in the sense of \cite{Ppa1994} and then show that the planar algebra of $N \subseteq K$ is
a $*$-planar subalgebra of that of $L \subseteq M$. In \S 4, we establish two `algebra to analysis' results. The first deals with going from an inclusion of finite pre-von Neumann algebras to the associated inclusion of their
von Neumann algebra completions. The second  deals with a sufficiently nice square of finite pre-von Neumann algebras and their corresponding completions, which are shown to form a commuting square. In \S 5, beginning with a $*$-planar subalgebra $Q$ of a subfactor planar algebra $P$, we appeal to the GJS construction to obtain a smooth non-degenerate commuting square, as above, such that the planar algebras of $N \subseteq K$ and $L \subseteq M$ are identified with $Q$ and $P$ respectively.

\section{Basics of GJS-construction}

In this section we recall the GJS construction from \cite{KdySnd2009}. Throughout this section, $P$ will be a subfactor planar algebra of modulus $\delta > 1$. The GJS construction associates to $P$, a basic construction tower $M_0 = N \subseteq M = M_1 \subseteq M_2 \subseteq \cdots$ of $II_1$-factors
with all inclusions extremal of finite index $\delta^2$	 and such that $P$ is the planar algebra of $N \subseteq M$. For all preliminary material on planar algebras we refer the reader to \cite{Jns1999} and for any unspecified notation to \cite{KdySnd2009}.

Let $P$ be a subfactor planar algebra of modulus $\delta >1.$  Therefore, we have finite-dimensional $C^*$-algberas
$P_n$ for $n\in Col$ with appropriate inclusions between them. For $k\geq 0$, let $F_k(P)$ be the vector space direct sum ${\oplus}_{n=k}^{\infty} P_n$ (where $0=0_{+}$, here and in the sequel)
. A typical element $a\in F_k(P)$ looks like $a=\sum_{n=k}^{\infty} a_n= (a_k,a_{k+1},\cdots)$, here of course
 only finitely many $a_n$'s are non-zero. According to \cite{KdySnd2009}, each $F_k(P)$ is equipped with a filtered, associative, unital $*$-algebra structure with normalised trace $t_k$ and there are trace preserving filtered $*$-algebra inclusions $F_0(P)\subseteq F_1(P)\subseteq F_2(P)\subseteq \cdots,$
as well as conditional expectation-like maps $F_k(P) \rightarrow F_{k-1}(P)$. 
Since we will need these structures explicitly in this paper, we briefly recall them.
%
First, note that an arbitrary element $a\in P_m \subseteq F_k(P)$ is depicted as in Figure \ref{fig:elt}:
\begin{figure}[!h]
\begin{center}
\psfrag{k}{\tiny $k$}
\psfrag{a}{\small $a$}
\psfrag{2n-2k}{\tiny $2m-2k$}
\includegraphics[scale=.55]{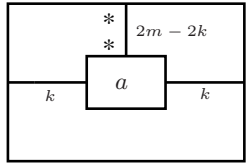}
\end{center}
\caption{Arbitrary element $a\in P_m \subseteq F_k(P)$}
\label{fig:elt}
\end{figure}
\bigskip
 
\textit{Multiplication:}  Consider two elements $a = a_m\in P_m\subseteq F_k(P)$ and $b = b_n\in P_n \subseteq F_k(P)$ for $m,n\geq k$. Then the mutiplication in $F_k(P)$, denoted by $(a\#b)$, is defined to be
$\sum_{t=\lvert n-m\rvert +k}^{n+m-k} (a\#b)_t$, where $(a\#b)_t$ is defined as in Figure \ref{fig:mult}.
\begin{figure}[!h]
\begin{center}
\psfrag{k}{\tiny $k$}
\psfrag{a}{\small $a$}
\psfrag{b}{\small $b$}
\psfrag{1}{\tiny $m+t-n-k$}
\psfrag{2}{\tiny $m+n-k-t$}
\psfrag{3}{\tiny $n+t-m-k$}
\includegraphics[scale=.7]{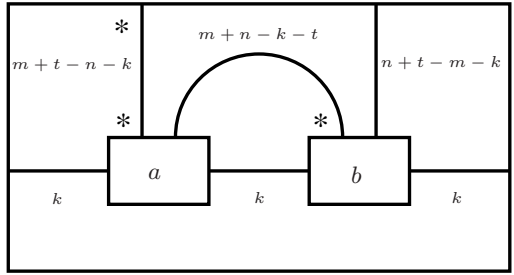}
\end{center}
\caption{Definition of the $P_t$ component of $a\#b$}
\label{fig:mult}
\end{figure}
\smallskip
As usual, extend the map bilinearly to the whole of $F_k(P)\times F_k(P)$. 
This multiplication map $\#$ makes each vector space $F_k(P)$ an associative, unital and filtered algebra.
\bigskip

\textit {Involution:} Define the involution map $\dagger: F_k(P) \rightarrow F_k(P)$ as follows. For $a=a_m\in P_m\subseteq F_k(P)$ define  $a^{\dagger}$ as in Figure \ref{fig:invol}. 
\begin{figure}[!h]
\begin{center}
\psfrag{k}{\tiny $k$}
\psfrag{a}{\small $a^*$}
\psfrag{2n-2k}{\tiny $2m-2k$}
\includegraphics[scale=.55]{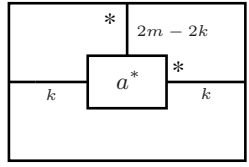}
\caption{Definition of the involution}
\label{fig:invol}
\end{center}
\end{figure}
In other words, $a^{\dagger}= Z_{R^k}(a^*)$. Here, of course $`*$' denotes the usual involution on $P_n$'s.
This involution map makes each $F_k(P)$ a $*$-algbera.
\bigskip

\textit{Trace:} For $a=(a_k,a_{k+1},\cdots)\in F_k(P)$ we define a linear functional $t_k$ on $F_k(P)$ by $t_k(a)=\tau(a_k)$
which defines a normalized trace on $F_k(P)$ that makes $\langle a, b\rangle := t_k(b^{\dagger}\# a)$ an inner-product on $F_k(P)$. 
Note that the trace of $a$ is the trace of its $P_k$-component.
\vspace{5mm}

\textit{Inclusion map:} $F_{k-1}(P)$ is included in $F_k(P)$ in such a way that the 
restriction takes $P_{m-1}\subseteq F_{k-1}(P)$ to $P_m\subseteq F_k(P)$ by taking $a\in P_{m-1}$ to the element in Figure \ref{fig:incl}.
\begin{figure}[!h]
\begin{center}
\psfrag{k}{\tiny $k-1$}
\psfrag{a}{\small $a$}
\psfrag{2n-2k}{\tiny $2m-2k$}
\includegraphics[scale=.55]{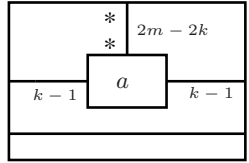}
\end{center}
\caption{Inclusion of  $a\in P_{m-1} \subseteq F_{k-1}(P)$ to $P_{m} \subseteq F_{k}(P)$}
\label{fig:incl}
\end{figure}

\bigskip

\textit{Conditional expectation-like map:} One defines a map $E_{k-1}: F_k(P)\mapsto F_{k-1}(P)$ (for $k\geq 1$) in such a way that
for any arbitrary $a\in P_m\subseteq F_k(P)$ the element $\delta E_{k-1}(a)$ of 
$P_{m-1}\subseteq F_{k-1}(P)$ is given by the tangle in Figure \ref{fig:condexp}.
\begin{figure}[!h]
\begin{center}
\psfrag{k}{\tiny $k-1$}
\psfrag{a}{\small $a$}
\psfrag{2n-2k}{\tiny $2m-2k$}
\includegraphics[scale=.55]{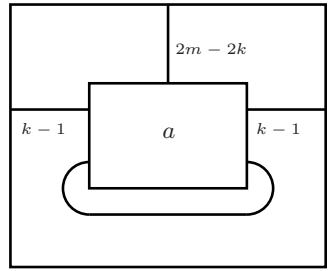}
\end{center}
\caption{Definition of $\delta E_{k-1}$}
\label{fig:condexp}
\end{figure}
\noindent Then the map $E_k$ is a $*$- and trace-preserving $F_k(P)-F_k(P)$ bimodule retraction for the inclusion map of $F_k(P)$ into $F_{k+1}(P)$.

\par We need a little bit more terminology.
\begin{definition}[\cite{KdySnd2009}]
 A \textbf{finite pre-von Neumann algebra} is a complex unital $*$-algebra $A$ that comes equipped with a normalized trace $t$ such that :
 \begin{itemize}
  \item the sesquilinear form defined by $\langle a,b\rangle = t(b^*a)$ defines an inner product on $A$. Denote the inner product by ${\langle. ,.\rangle}_A$.
  \item for each $a\in A$, the left multiplication map $\lambda_A(a): A\rightarrow A$ is bounded for the trace-induced norm of $A$.
 \end{itemize}

\end{definition}

Examples are $F_k(P)$ with their natural traces $t_k$ for $k\geq 0$.

\begin{notation}\label{completion}
 Let $\mathcal{H}_A$ be the Hilbert space completion of $A$ for the associated norm. As usual there exists a natural one-one and linear map $\Gamma: A\rightarrow \mathcal{H}_A$ such that ${\langle \Gamma(a), \Gamma(b)\rangle}_{{\mathcal{H}}_{A}} = {\langle a, b\rangle}_{A}$ for all $a,b\in A$ and $\Gamma(A)$ is dense in $\mathcal{H}_{A}.$
Let us denote by $\mathcal{H}_k(P)$ the Hilbert space completion of $F_k(P)$. 
 The Hilbert space $\mathcal{H}_k(P)$ is nothing but the orthogonal direct sum ${\oplus}_{n=k}^{\infty}P_n$.
\end{notation}

The following two results that we quote without proof are taken from \cite{KdySnd2009}.

\begin{lemma}[Lemma 4.4 of \cite{KdySnd2009}]\label{l1}
 Let $A$ be a finite pre-von Neumann algebra with trace $t_A$, and $\mathcal{H}_{A}$ be the Hilbert space completion of $A$ for the associated norm, so that
 the left regular representation $\lambda_A:A\rightarrow \mathcal{B}({\mathcal{H_A}})$ is well defined, i.e. for each $a\in A$, $\lambda_A(a):A\rightarrow A$ extends to a bounded operator on $\mathcal{H}_{A}$.
 Let $M^{}_A = {\lambda_A(A)}^{\dprime}.$ Then, 
 \begin{enumerate}
  \item  The `vacuum vector' $\Omega_A\in \mathcal{H}_A$(corresponding to $1\in A\subseteq \mathcal{H}_{A})$ is cyclic and separating for the von-Neumann algebra  $M^{}_A$.
  \item The trace $t_A$ extends to faithful, normal, tracial states $t^{}_A$ on $M^{}_A$.

 \end{enumerate}
 \end{lemma}
 \smallskip
 
\par Let us denote by $M^{}_k(P) \subseteq \mathcal{B}(\mathcal{H}_k(P))$  the von Neumann algebra corresponding to the finite pre-von Neumann algebra $F_k(P)$. Let $\Omega_k \in \mathcal{H}_k(P)$ be the cyclic and separating vector for $M^{}_k(P)$. It is easy to see that under the appropriate identifications, 
$\Omega_0 = \Omega_1 = \Omega_2 = \cdots$.

\begin{theorem}[Theorem 6.2 of \cite{KdySnd2009}]\label{gjsks}
 Let $P$ be a subfactor planar algebra of modulus $\delta>1$. Then $M^{}_{k-1}(P)\subseteq M^{}_k(P)\subseteq M^{}_{k+1}(P)$ is (isomorphic to) a basic construction tower of type $II_1$ -factors with finite index ${\delta}^2$.
 Moreover, the subfactor $M^{}_0(P) \subseteq M^{}_1(P)$ constructed from $P$ is a finite index and extremal subfactor with planar algebra isomorphic to $P$.
 \end{theorem}

\section{Commuting squares to planar subalgebras}
Throughout this section, we will deal with a non-degenerate commuting square $\mathscr{C}$ of $II_1$-factors
$$
\begin{matrix}
L &\subseteq & M\cr
\rotatebox{90}{$\subseteq$} &\ &\rotatebox{90}{$\subseteq$}\cr
N &\subseteq & K
\end{matrix}
$$ 
with all inclusions extremal of finite index. Suppose that the associated basic construction towers are given by
\[
\begin{array}{ccccccccccc}
M_0 = L & \ss & M = M_1 & \ss & M_2 & \ss &
\cdots  \\
\rotatebox{90}{$\subseteq$} &&\rotatebox{90}{$\subseteq$} && \rotatebox{90}{$\subseteq$} && \\
K_0 = N & \ss & K = K_1 & \ss & K_2 & \ss &
\cdots  \\
\end{array}\]

%
The following definition is from \cite{Ppa1994}.

\begin{definition}\label{s}
The commuting square $\mathscr{C}$ is said to be smooth if 
$$N^{\prime}\cap K_k \subseteq L^{\prime}\cap M_k$$ for all $k\geq 0$.
\end{definition}

Before we prove the main result of this section, we note a result of \cite{Ppa1994} -
see Proposition 2.3.2 - which implies that if $\mathscr{C}$ is smooth, then each square in the basic construction tower
 is also a smooth non-degenerate commuting square.
We also recall without proof two elementary facts about general  commuting squares (i.e., not necessarily non-degenerate and not necessarily factors).

\begin{lemma}
 \label{coset}
 Let $A_{10} \subseteq A_{11}$ be a pair of finite von Neumann algebras with a faithful normal trace $tr$ on $A_{11}$ and let $S$ be a self-adjoint subset of $A_{10}$. Then
 $$\begin{matrix}
A_{10} &\subseteq & A_{11}\cr
\rotatebox{90}{$\subseteq$} &\ &\rotatebox{90}{$\subseteq$}\cr
S^{\prime}\cap A_{10} &\subseteq & S^{\prime}\cap A_{11}
\end{matrix}
$$
is a commuting square.\qed
\end{lemma}

\begin{lemma}
 \label{comm1}
 Consider a tower of quadruples of finite von Neumann algebras with a faithful normal trace $tr$ on $A_{12}$
 $$
\begin{matrix}
A_{10} &\subseteq & A_{11} & \subseteq & A_{12}\cr
\rotatebox{90}{$\subseteq$} &\ &\rotatebox{90}{$\subseteq$} &\ & \rotatebox{90}{$\subseteq$}\cr
A_{00} &\subseteq & A_{01} & \subseteq & A_{02}
\end{matrix}
$$ such that the following two squares are commuting squares with respect to $tr$
$$
\begin{matrix}
A_{10} &\subseteq & A_{12}\cr
\rotatebox{90}{$\subseteq$} &\ &\rotatebox{90}{$\subseteq$}\cr
A_{00} &\subseteq & A_{02}
\end{matrix} \hspace{0.5cm}\text{and}\hspace{0.5cm}
\begin{matrix}
A_{11} &\subseteq & A_{12}\cr
\rotatebox{90}{$\subseteq$} &\ &\rotatebox{90}{$\subseteq$}\cr
A_{01} &\subseteq & A_{02}
\end{matrix}.
$$ 
Then, $$
\begin{matrix}
A_{10} &\subseteq & A_{11}\cr
\rotatebox{90}{$\subseteq$} &\ &\rotatebox{90}{$\subseteq$}\cr
A_{00} &\subseteq & A_{01}
\end{matrix}
$$ is also a commuting square.\qed
\end{lemma}

 \begin{theorem}
 \label{commuting}
If $\mathscr{C}$ is a smooth non-degenerate commuting square, then the planar algebra of $(N\subseteq K)$ is a planar subalgebra of the planar algebra of $(L \subseteq M)$.
\end{theorem}

\begin{proof}

First, denote by $P = P^{(N\subseteq K)}$ (respectively, $Q = P^{(L \subseteq M)}$) the planar algebra of $N\subseteq K$ (respectively, $L \subseteq M)$. Figure \ref{fig:cubical} shows the standard invariants of $N \ss K$ and $L \ss M$. 
Each arrow
in this figure represents an inclusion map. The dotted arrows on the top level are well defined maps
by the assumption of smoothness on $\mathscr{C}$ while those on the bottom level
are so because of Proposition 2.3.2 of \cite{Ppa1994}.

In particular, for $n \geq 0$, the spaces $P_n = N^\prime \cap K_n$ of the planar algebra $P$ are subspaces of the spaces $Q_n = L^\prime \cap M_n$ of the planar algebra $Q$ - as
observed from the dotted arrows on the top level. 
Also $P_{0_-} = K^\prime \cap K_0$ is the whole of $Q_{0_-} = M^\prime \cap M_0$
-  both being ${\mathbb C}$.
Hence to see that $P$ is a planar subalgebra of $Q$, it suffices to see that for any tangle $T= T^{k_0}_{k_1,\cdots,k_b}$ in a class of `generating tangles', and inputs $x_i \in P_{k_i}$, $Z_T^Q(x_1 \otimes \cdots \otimes x_b) \in P_{k_0}$.

We will use the following collection of generating tangles - see Theorem 3.3 of \cite{KdySnd2004} but with notation for the tangles as in \cite{KdySnd2009} -
$1^{0_+},1^{0_-}$, $E^{n+2} {\text {~for $n \geq 0$}}$, $EL(1)^{n+1}_{n+1} {\text {~for $n \geq 0$}}$, $I_n^{n+1} {\text {~for $n \in Col$}}$, $ER_{n+1}^n {\text {~for $n \in Col$}}$ and $M_{n,n}^n$ $ {\text {for $n \in Col$}}$. In order to be self-contained we illustrate these tangles in Figures \ref{fig:tang1} and \ref{fig:tang2}. Note that the tangles $1^{0_+}$ and $1^{0_-}$ are identical except for the shading which is omitted.

\begin{figure}[!htb]
\centering
\psfrag{n}{\tiny $n$}
\psfrag{i}{\tiny $i$}
\psfrag{$ER_{n+i}^n$ : Right expectations}{$ER_{n+1}^n$ : Right expectation}
\psfrag{2n}{\tiny $2n$}
\psfrag{2}{\tiny $2$}
\psfrag{D_1}{\small $D_1$}
\psfrag{D_2}{\small $D_2$}
\psfrag{$qq$}{$1^{0_\pm}$ : Unit tangles}
\psfrag{$I_n^{n+1}$ : Inclusion}{$I_n^{n+1}$ : Inclusion}
\psfrag{$EL(i)_{n+i}^{n+i}$ : Left expectations}{$EL(1)_{n+1}^{n+1}$ : Left expectation}
\psfrag{$M_{n,n}^n$ : Multiplication}{$M_{n,n}^n$ : Multiplication}
\psfrag{$TR_n^0$ : Trace}{$TR_n^0$ : Trace}
\psfrag{$R^{n+1}_{n+1}$ : Rotation}{$R_{n+1}^{n+1}$ : Rotation}
\psfrag{$1^n$ : Multiplicative identity}{$1^n$ : Mult. identity}
\psfrag{$ER_{n+1}^{n+1}$ : Right expectation}{$ER_{n+1}^{n+1}$ : Right expectation}
\psfrag{$I_n^n$ : Identity}{$I_n^n$ : Identity}
\psfrag{$E^{n+2}$ : Jones projections}{$E^{n+2}$ : Jones projection}
\includegraphics[height=7cm]{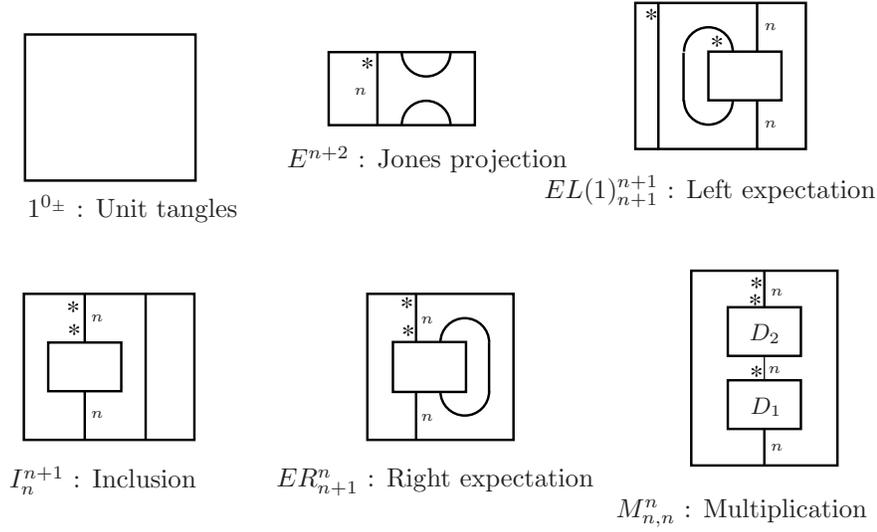}
\caption{$1^{0_\pm}, E^{n+2}, EL(1)^{n+1}_{n+1}, I_n^{n+1}, ER_{n+1}^n, M_{n,n}^n$ for $n \geq 0$.}
\label{fig:tang1}
\end{figure}



\begin{figure}[!htb]
\centering
\psfrag{n}{\tiny $n$}
\psfrag{i}{\tiny $i$}
\psfrag{$ER_{n+i}^n$ : Right expectations}{$ER_{n+1}^n$ : Right expectation}
\psfrag{2n}{\tiny $2n$}
\psfrag{2}{\tiny $2$}
\psfrag{D_1}{\small $D_1$}
\psfrag{D_2}{\small $D_2$}
\psfrag{$qq$}{$I_{0_-}^1$}
\psfrag{$I_n^{n+1}$ : Inclusion}{$I_n^{n+1}$ : Inclusion}
\psfrag{$EL(i)_{n+i}^{n+i}$ : Left expectations}{$EL(1)_{n+1}^{n+1}$ : Left expectation}
\psfrag{$M_{n,n}^n$ : Multiplication}{$M_{0_-,0_-}^{0_-}$}
\psfrag{$TR_n^0$ : Trace}{$TR_n^0$ : Trace}
\psfrag{$R^{n+1}_{n+1}$ : Rotation}{$R_{n+1}^{n+1}$ : Rotation}
\psfrag{$1^n$ : Multiplicative identity}{$1^n$ : Mult. identity}
\psfrag{$ER_{n+1}^{n+1}$ : Right expectation}{$ER_1^{0_-}$}
\psfrag{$I_n^n$ : Identity}{$I_n^n$ : Identity}
\psfrag{$E^{n+2}$ : Jones projections}{$E^{n+2}$ : Jones projection}
\includegraphics[height=2.5cm]{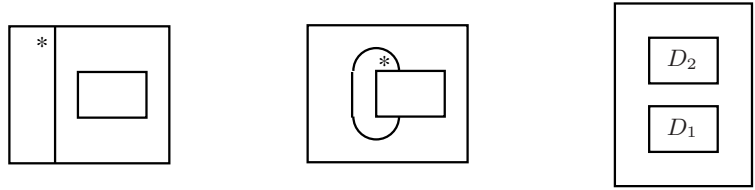}
\caption{$I_n^{n+1}, ER_{n+1}^n$ and $M_{n,n}^n$ for $n = 0_-$.}
\label{fig:tang2}
\end{figure}

%
%
%

\begin{figure}[!h]
\centering
\psfrag{1}{\tiny $L' \cap M_0$}
\psfrag{2}{\tiny $L' \cap M_1$}
\psfrag{3}{\tiny $L' \cap M_2$}
\psfrag{4}{\tiny $L' \cap M_3$}
\psfrag{5}{\tiny $N' \cap K_0$}
\psfrag{6}{\tiny $N' \cap K_1$}
\psfrag{7}{\tiny $N' \cap K_2$}
\psfrag{8}{\tiny $N' \cap K_3$}
\psfrag{9}{\tiny $M' \cap M_1$}
\psfrag{10}{\tiny $M' \cap M_2$}
\psfrag{11}{\tiny $M' \cap M_3$}
\psfrag{12}{\tiny $K' \cap K_1$}
\psfrag{13}{\tiny $K' \cap K_2$}
\psfrag{14}{\tiny $K' \cap K_3$}
\psfrag{15}{\tiny $K' \cap K_0$}
\psfrag{16}{\tiny $M' \cap M_0$}
\includegraphics[scale=0.6]{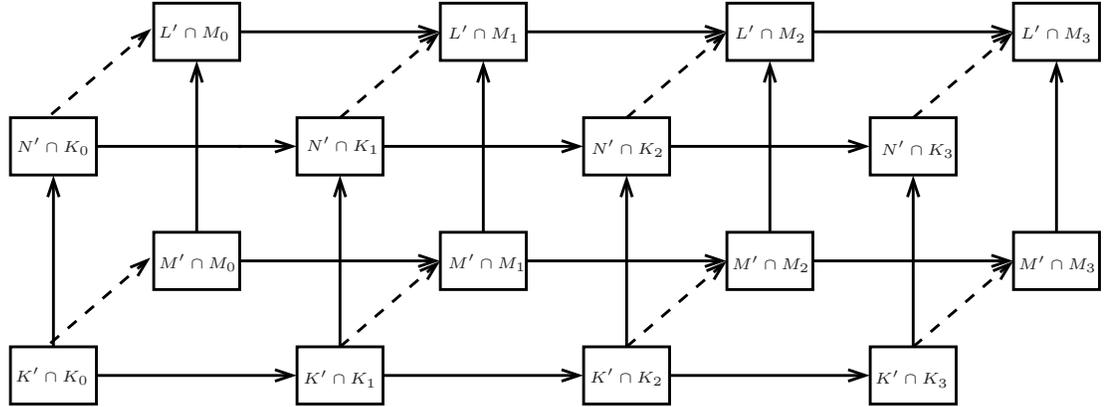}
\caption{Standard invariants of $N \ss K$ and $L \ss M$}
\label{fig:cubical}
\end{figure}

We begin by observing that since $P_n$ are unital subalgebras of $Q_n$ and form
an increasing chain, if $T$ is one of the tangles $1^{0_\pm}$, $I_n^{n+1}$ or $M_{n,n}^n$ then the output for $T$ lies in $P$ whenever the inputs do. We will now
verify that this holds for the remaining three generating tangles  $E^{n+2} {\text {~for $n \geq 0$}}$, $EL(1)^{n+1}_{n+1} {\text {~for $n \geq 0$}}$,  $ER_{n+1}^n {\text {~for $n \in Col$}}$.

Case I: $T = E^{n+2}$ for $n \geq 0$: What needs to be seen is that $Z^Q_T(1)$
lies in $P_{n+2}$. However, $Z^Q_T(1)$ is a scalar multiple  of the Jones projection for the inclusion
$M_n \subseteq M_{n+1}$ which also is the Jones projection for the inclusion 
$K_n \subseteq K_{n+1}$ (since these form a non-degenerate commuting square) and hence lies in $P_{n+2} = N^\prime \cap K_{n+2}$.

Case II: $T = EL(1)^{n+1}_{n+1} {\text {~for $n \geq 0$}}$: What needs to be seen is that for $x \in P_{n+1} = N^\prime \cap K_{n+1}$, $Z^Q_T(x)$ also lies in $P_{n+1}$.
A moment's thought shows that this will follow if the `side faces' of the cubes in Figure \ref{fig:cubical} are commuting squares. Thus we need to see that for any $n \geq 0$,
the square 
$$
\begin{matrix}
N^\prime \cap K_{n+1} &\subseteq & L^\prime \cap M_{n+1}\cr
\rotatebox{90}{$\subseteq$} &\ &\rotatebox{90}{$\subseteq$}\cr
K^\prime \cap K_{n+1} &\subseteq & M^\prime \cap M_{n+1}\end{matrix}
$$ 
is a commuting square.

In other words, we need to show that for any $x\in M^{\prime}\cap M_{n+1}$, we must have $E^{L^{\prime}\cap M_{n+1}}_{N^{\prime}\cap K_{n+1}}(x)\in K^{\prime}\cap K_{n+1}$.
First observe that the following quadruple, say $\mathcal{D}$, is a commuting square:

$$
\begin{matrix}
K^\prime \cap M_{n+1} &\subseteq & N^\prime \cap M_{n+1}\cr
\rotatebox{90}{$\subseteq$} &\ &\rotatebox{90}{$\subseteq$}\cr
K^\prime \cap K_{n+1} &\subseteq & N^\prime \cap K_{n+1}\end{matrix}
$$ 
Indeed, this follows once we apply Lemma \ref{coset} and Lemma \ref{comm1} to the following tower of quadruples:

$$
\begin{matrix}
K^\prime \cap M_{n+1} &\subseteq & N^\prime \cap M_{n+1} & \subseteq & M_{n+1}\cr
\rotatebox{90}{$\subseteq$} &\ &\rotatebox{90}{$\subseteq$} &\ &\rotatebox{90}{$\subseteq$} \cr
K^\prime \cap K_{n+1} &\subseteq & N^\prime \cap K_{n+1} & \subseteq & K_{n+1} \end{matrix}.
$$ 
Now  since $x\in  M^{\prime}\cap M_{n+1}\subset K^{\prime}\cap M_{n+1}$ and $\mathcal{D}$ is a commuting square it follows that $E_{N^{\prime}\cap K_{n+1}}^{N^{\prime}\cap M_{n+1}}(x)\in K^{\prime}\cap K_{n+1}$.
This completes the proof of case II.

%

Case III: $T = ER_{n+1}^n {\text {~for $n \in Col$}}$: If $n = 0_-$ the verification is trivial so we will treat the case $n \geq 0$.
What needs to be seen is that for $x \in P_{n+1} = N^\prime \cap K_{n+1}$, $Z^Q_T(x)$  lies in $P_{n}= N^{\prime}\cap K_n$.
A moment's thought shows that this will follow if the `top faces' of the cubes in Figure \ref{fig:cubical} are commuting squares. Thus we need to see that for any $n \geq 0$,
the square $\mathcal{O}$
$$
\begin{matrix}
L^\prime \cap M_{n} &\subseteq & L^\prime \cap M_{n+1}\cr
\rotatebox{90}{$\subseteq$} &\ &\rotatebox{90}{$\subseteq$}\cr
N^\prime \cap K_{n} &\subseteq & N^\prime \cap K_{n+1}\end{matrix}
$$ 
is a commuting square. Consider $x\in N^{\prime}\cap K_{n+1}$. By Lemma \ref{coset} $$
\begin{matrix}
M_n &\subseteq & M_{n+1}\cr
\rotatebox{90}{$\subseteq$} &\ &\rotatebox{90}{$\subseteq$}\cr
L^{\prime}\cap M_n &\subseteq & L^{\prime}\cap M_{n+1}
\end{matrix}
$$ 
is a commuting square. Thus, $E^{L^{\prime}\cap M_{n+1}}_{L^{\prime}\cap M_n}(x)=E^{M_{n+1}}_{M_n}(x)$. Since, $$\begin{matrix}
M_n &\subseteq & M_{n+1}\cr
\rotatebox{90}{$\subseteq$} &\ &\rotatebox{90}{$\subseteq$}\cr
K_n &\subseteq &  K_{n+1}
\end{matrix}
$$ is also a commuting square we immediately obtain $E^{M_{n+1}}_{M_n}(x)=E^{K_{n+1}}_{K_n}(x).$ Again, by Lemma \ref{coset} we see that 
$$\begin{matrix}
K_n &\subseteq & K_{n+1}\cr
\rotatebox{90}{$\subseteq$} &\ &\rotatebox{90}{$\subseteq$}\cr
N^{\prime}\cap K_n &\subseteq & N^{\prime}\cap K_{n+1}
\end{matrix}
$$ is a commuting square and hence $E^{K_{n+1}}_{K_n}(x)=E^{N^{\prime}\cap K_{n+1}}_{N^{\prime}\cap K_n}(x).$ Therefore, 
$$E^{L^{\prime}\cap M_{n+1}}_{L^{\prime}\cap M_n}(x)=E^{N^{\prime}\cap K_{n+1}}_{N^{\prime}\cap K_n}(x).$$ This proves that $\mathcal{O}$ is a commuting square as desired.\end{proof}

\color{black}
\section{Compatible pairs and quadruples  of finite pre-von Neumann algebras}

\begin{definition}
 \cite{KdySnd2009} A \textbf{compatible pair of finite pre-von Neumann algebras} is a pair $(A,t_A)$ and $(B,t_B)$ of finite pre-von Neumann algebras such that $A\subseteq B$ is a unital inclusion and 
 $t_B{\big|}_A= t_A.$ Given a such pair of compatible pre-von Neumann algebras, identify $\mathcal{H}_A$ with a subspace of $\mathcal{H}_B$ so that $\Omega_A = \Omega_B = \Omega.$
\end{definition}
Some of the following results may be implicit in \cite{KdySnd2009}.
\begin{theorem}\label{t1}
 Let $(A,t_A) \subseteq (B,t_B)$ be a compatible pair of finite pre-von Neumann algebras. Let $E_A:B\rightarrow A$ be a $*$-and trace-preserving $A-A$ bimodule retraction for the inclusion map of $A$ into $B$.
Let $\lambda_A:A\rightarrow \mathcal{B}({\mathcal{H}_A})$ and $\lambda_B:B\rightarrow \mathcal{B}({\mathcal{H}_B})$ be the left regular representations of $A$ and $B$ respectively and let $M^{}_A = {\lambda_A(A)}^{\dprime}$
and $M^{}_B= {\lambda_B(B)}^{\dprime}.$ Then, 
\begin{enumerate}
 \item $t^{\lambda}_B({\iota}_A(x))=t^{\lambda}_A(x)$ for $x\in M^{}_A$, where ${\iota}_A$ is the normal inclusion of $M^{}_A$ into $M^{}_B$ as in Proposition 4.6 of  \cite{KdySnd2009}.
 \item The map $E_A$ extends continuously to an orthogonal projection, call it $e_A$, from $\mathcal{H}_{B}$ onto the closed subspace $\mathcal{H}_{A}$.
 \item $J_Be_AJ_B=e_A$, where $J_B$ is the modular conjugation operator which is the unique bounded extension, to $\mathcal{H}_{B}$, of the involutive,
 conjugate-linear, isometry defined on the dense subspace $\Gamma(B)\subseteq \mathcal{H}_B$ by $\Gamma(b)\mapsto \Gamma(b^*)$.
 \item The map $E_A$ extends to the unique trace preserving conditional expectation map $E^{}_A: M^{}_B\rightarrow M^{}_A$. It is continuous for the $\text{SOT}^*$-topologies on the domain and range.
 \item $E^{}_A(x)(\Gamma(a))=J_A(\lambda_A(a^*)e_Ax^*\Omega)$ for $x\in M^{}_B$ and $a\in A.$
 \end{enumerate}

 \end{theorem}
 \begin{proof}
  (1) First recall (see the proof of Lemma \ref{l1}) that for a finite pre-von Neumann algebra $B$, the faithful, normal, tracial state $t^{\lambda}_B$ is the restriction to the von Neumann algebra $M^{}_B$ of the linear functional $\widetilde{t_B}$
  on $\mathcal{B}({\mathcal{H}_B})$ defined  by $\widetilde{t_B}(x)=\langle x\Omega, \Omega\rangle$ for $x\in \mathcal{B}({\mathcal{H}_B})$. Note that $\widetilde{t_B}(\lambda_B(b))=t_B(b)$ for $b\in B$. Now, for any
  $x\in M^{}_A$ we have, $t^{\lambda}_B({\iota}_A(x))= \widetilde{t_B}({\iota}_A(x))= \langle {\iota}_A(x)\Omega,\Omega\rangle = \langle J_B(x^*\Omega),\Omega\rangle=\langle J_B\Omega, x^*\Omega\rangle= \langle x\Omega,\Omega\rangle = \widetilde{t_A}(x)= t^{\lambda}_A(x)$.
      
  (2) The continuous extension from $\mathcal{H}_B $ onto $\mathcal{H}_A$ of $E_A$ is defined by $e_A(\Gamma(b))=\Gamma(E_A(b))$ for $b\in B.$ We show that $e_A$ is nothing but the orthogonal projection onto the closed
  subspace $\mathcal{H}_A$ of $\mathcal{H}_B$. Observe that the range of the operator $e_A\in \mathcal{B}(\mathcal{H}_B)$ is precisely $\mathcal{H}_A.$ Since, $E_A$ is a $A-A$ bimodule retraction map for the inclusion of $A$ into $B$ it is clear that $e_A^2=e_A.$ 
Another routine calculation proves that  ${\langle e_A(\Gamma(b_1)),\Gamma(b_2)\rangle}_{\mathcal{H}_B} = {\langle \Gamma(b_1), e_A(\Gamma(b_2))\rangle}_{\mathcal{H}_B}$ and hence $e_A^*=e_A.$ This proves that $e_A$ is the orthogonal projection 
onto $\mathcal{H}_A$.

(3) First it should be clear that $J_B{\big |}_A= J_A.$ Now for any $b\in B$ we have  $J_Be_AJ_B(\Gamma(b))= J_Be_A(\Gamma(b^*))= J_B\Gamma(E_A(b^*))=\Gamma(E_A(b))= e_A(\Gamma(b)).$ Since $J_Be_AJ_B$ and $e_A$ agree
on a dense set we get the desired equality.

(4) To obtain a formula for a condition expectation from $M^{}_B$ onto $M^{}_A$ we follow the standard trick. As a first step we prove that
for any $x\in M^{}_B, e_Axe_A\in M^{}_Ae_A.$ The reader should observe that $e_A\in \big({M^{}_A}\big)^{\prime}$. Further, $J_AM^{}_AJ_A= (M^{}_A)^{\prime}$ (see the proof of Lemma 4.4 (item (2)) of \cite{KdySnd2009}).
Therefore, it is sufficient to prove that $ e_Axe_A\in (J_BM^{}_AJ_B)^{\prime}e_A.$ Indeed, it is routine to check that 
$ e_Axe_A\in (e_AJ_BM^{}_AJ_B e_A)^{\prime}$ and hence the conclusion follows. Suppose, $e_Axe_A = \widetilde{E}(x)e_A$ for some $\widetilde{E}(x)\in M^{}_A.$ Then, 
$\widetilde{E}(x)\in M^{}_A$ is uniquely determined since $\Omega$ is separating for $M^{}_A$ by Lemma \ref{l1}.
Define, $E^{}_A:M^{}_B\rightarrow M^{}_A$ by $E^{}_A(x)= \widetilde{E}(x)$ for $x\in M^{}_B.$ Next we show that for $y\in M^{}_A$, $t^{\lambda}_A(\widetilde{E}(x)y)=t^{\lambda}_B(x{\iota}_A(y)).$
This follows from the following array of equations.
\begin{align*}
  t^{\lambda}_A(\widetilde{E}(x)y)  & \qquad = \widetilde{t_A}(\widetilde{E}(x)y)\\
 &\qquad = {\langle \widetilde{E}(x)y\Omega,\Omega\rangle}_A\\
 &\qquad = {\langle \widetilde{E}(x)e_Ay\Omega,\Omega\rangle}_A\\
 &\qquad= {\langle e_Axe_Ay\Omega,\Omega\rangle}_B~~~~~~~~~~~~~~~[\textrm{Since}~~e_Axe_A=\widetilde{E}(x)e_A]\\
 &\qquad= {\langle x(y\Omega),\Omega\rangle}_B\\
 &\qquad= t^{\lambda}_B(x{\iota}_A(y)).
\end{align*}
Thus by \cite{Mgk1954} $E^{}_A$ is the unique  trace preserving conditional expectation from $M^{}_B$ onto $M^{}_A$ with respect to the trace $t^{\lambda}_B$.
We want to show that the conditional expectation map $E^{}_A:M^{}_B \mapsto M^{}_A$  is continuous for the $\text{SOT}^*$-topologies on the domain and range. Suppose, a net $\{x_{\alpha}\}$ converges to $x\in M^{}_B$ in $\text{SOT}^*$-topology. 
Take an arbitrary element $\xi\in \mathcal{H}_A$. There exists $\eta\in \mathcal{H}_B$ such that $e_A\eta=\xi.$ Now observe that, $E^{}_A(x_{\alpha})\xi=E^{}_A(x_{\alpha})e_A\eta=e_Ax_{\alpha}e_A\eta$. 
But since $x_{\alpha}$ converges to $x$ in $\text{SOT}^*$ topology we see that $e_Ax_{\alpha}e_A\eta$ converges to $e_Axe_A\eta$.
In other words, the map $E^{}_A(x_{\alpha})\xi$ converges to $E^{}_A(x)\xi$. Now the continuity of $E^{}_A$ in $\text{SOT}^*$ topology follows from the fact that $E^{}_A(x^*) = \big(E^{}_A(x)\big)^*.$

(5) Define 
$P= \{x\in M^{}_B: E^{}_A(x)(\Gamma(a))=J_A(\lambda_A(a^*)e_Ax^*\Omega) \forall a\in A\}.$
Simple calculations show that for each $b\in B$, $\lambda_B(b) \in P$.
In fact, for any $a \in A$, $J_A(\lambda_A(a^*)e_A\lambda_B(b)^*\Omega) = \Gamma(E_A(b)a) = 
E^{}_A(\lambda_B(b))(\Gamma(a))$.
\par Next we show that $P$ is a $\text{SOT}^*$ closed subspace of $M^{}_B$. For this, consider a net $\{x_{\alpha}\}\subseteq P$ converging to $x$ in $\text{SOT}^*$ topology.
As we have already seen that  $E^{}_A$ is $\text{SOT}^*$ continuous, $E^{}_A(x_{\alpha})$ also converges to $E^{}_A(x)$ in $\text{SOT}^*$ topology. Thus,
$E^{}_A(x_{\alpha})(\Gamma(a))\longrightarrow E^{}_A(x)(\Gamma(a))$. But since each $x_{\alpha}$ belongs to $P$ we see that $E^{}_A(x_{\alpha})(\Gamma(a))=
J_A(\lambda_A(a^*)e_A x^*_{\alpha}\Omega).$
But as $\{x_{\alpha}\}$ converges to $x$ in $\text{SOT}^*$ topology we see that $J_A(\lambda_A(a^*)e_Ax^*_{\alpha}\Omega)\longrightarrow J_A(\lambda_A(a^*)e_Ax^*\Omega).$ Therefore, $E^{}_A(x)(\Gamma(a))= J_A(\lambda_A(a^*)e_Ax^*\Omega).$
Thus we have proved that $P$ is a $\text{SOT}^*$ closed subspace of $M^{}_B$. Furthermore, $\lambda_B(B)\subseteq P.$ Hence, $P=M^{}_B.$
This completes the proof.
 \end{proof}
 Let us record two easy facts  for future reference.
 \begin{remark}\label{r1}
The inclusion map $\iota_A$ (of $M^{}_A$ into $M^{}_B$) is given by the formula $\iota_A(x)(\Gamma(b))=J_B\big(\lambda_B(b^*)x^*\Omega\big)$ for $x\in M^{}_A$ and $b\in B$. Therefore, $\iota_A(\lambda_A(a))=\lambda_B(a)$ for any $a\in A$. See \cite{KdySnd2009} for details.
\end{remark}
\begin{remark}\label{r2}
For any $b\in B$, $E^{}_A(\lambda_B(b))=\lambda_A(E_A(b)).$ This can be easily verified using the formula of $E^{}_A$ given in Theorem \ref{t1}.
  \end{remark}

\begin{definition}
 A \textbf{commuting square of finite pre-von Neumann algebras} is a quadruple \[\begin{array}{ccc}
  A_{10} & \subseteq &  A_{11}\\
\rotatebox{90}{$\subseteq$} & & \rotatebox{90}{$\subseteq$}\\
  A_{00} & \subseteq & A_{01}
\end{array}\] satisfying the following three properties:
\begin{itemize}
 \item each pair of inclusions in the quadruple is a compatible pair of finite pre-von Neumann algebras; that is, $t_{A_{11}}{\big{|}}_{A_{ij}}= t_{A_{ij}}$ for $i,j\in \{0,1\}$.
 \item there exist $*$-and trace-preserving $A_{ij}-A_{ij}$ bimodule retractions $E_{A_{ij}}$, corresponding to each $i,j\in\{0,1\}$, for the inclusion map of $A_{ij}$ into $A_{11}$.
 \item $E_{A_{10}} E_{A_{01}}(a_{11})=E_{A_{00}}(a_{11})= E_{A_{01}} E_{A_{10}}(a_{11})$ for $a_{11} \in A_{11}.$
 \end{itemize}
\end{definition}
\begin{theorem}\label{lambda} Consider a commuting square of finite pre-von Neumann algebras \[\begin{array}{ccc}
  A_{10} & \subseteq &  A_{11}\\
\rotatebox{90}{$\subseteq$} & & \rotatebox{90}{$\subseteq$}\\
  A_{00} & \subseteq & A_{01}
\end{array}.\]
 Following the notation of Theorem \ref{t1} the quadruple \[\begin{array}{ccc}
  M^{}_{A_{10}} & \subseteq &  M^{}_{A_{11}}\\
\rotatebox{90}{$\subseteq$} & & \rotatebox{90}{$\subseteq$}\\
  M^{}_{A_{00}} & \subseteq & M^{}_{A_{01}}
\end{array}\] is a commuting square of von Neumann algebras with respect to the inclusions ${\iota}_{A_{ij}}$ of $M^{}_{A_{ij}}$ into $M^{}_{A_{11}}$ and conditional expectations
$E^{}_{A_{ij}}:M^{}_{A_{11}}\rightarrow M^{}_{A_{ij}} .$
\end{theorem}

\begin{proof}
 To see that the  quadruple of von Neumann algebras is a commuting square the equation needed to be verified (for any $x\in M^{}_{A_{11}}$)  is the following:
 $${\iota}_{A_{10}} E^{}_{A_{10}} {\iota}_{A_{01}} E^{}_{A_{01}}(x) = {\iota}_{A_{00}}E^{}_{A_{00}}(x)= {\iota}_{A_{01}} E^{}_{A_{01}}{\iota}_{A_{10}} E^{}_{A_{10}}(x).$$
 Define, $Q= \{x\in  M^{}_{A_{11}}:{\iota}_{A_{10}} E^{}_{A_{10}} {\iota}_{A_{01}} E^{}_{A_{01}}(x) = {\iota}_{A_{00}}E^{}_{A_{00}}(x)= {\iota}_{A_{01}} E^{}_{A_{01}} {\iota}_{A_{10}} E^{}_{A_{10}}(x) \}.$

  Now using Remarks \ref{r1} and Remarks \ref{r2} it follows easily that for any $\lambda_{A_{11}}(a_{11})\in \lambda_{A_{11}}(A_{11})\subseteq M^{}_{A_{11}}$ the following equations hold:
\begin{eqnarray*}
{\iota}_{A_{10}} E^{}_{A_{10}}{\iota}_{A_{01}}E^{}_{A_{01}}(\lambda_{A_{11}}(a_{11})) &=& \lambda_{A_{11}}(E_{A_{10}}E_{A_{01}}(a_{11})), \\
{\iota}_{A_{01}} E^{}_{A_{01}}{\iota}_{A_{10}}E^{}_{A_{10}}(\lambda_{A_{11}}(a_{11})) &=& \lambda_{A_{11}}(E_{A_{01}}E_{A_{10}}(a_{11})), {\text {and}}\\
{\iota}_{A_{00}}E^{}_{A_{00}}({\lambda}_{A_{11}}(a_{11})) &=& {\lambda}_{A_{11}}(E_{A_{00}}(a_{11})).
\end{eqnarray*}
Since by assumption we have that $E_{A_{10}} E_{A_{01}}(a_{11})=E_{A_{00}}(a_{11})= E_{A_{01}} E_{A_{10}}(a_{11})$ for $a_{11}\in A_{11}$
we conclude that ${\lambda}_{A_{11}}(A_{11}) \subseteq Q.$ Furthermore, since each ${\iota}_{A_{ij}}$ and each $E^{}_{A_{ij}}$ is $\text{SOT}^*$-continuous we conclude that
$Q$ is an $\text{SOT}^*$-closed subspace of $M_{A_{11}}$. Therefore, $Q=M^{}_{A_{11}}$. This completes the proof.
\end{proof}

\section{Planar subalgebras and commuting squares}
In this section we assume that $Q$ is a $*$-planar subalgebra of a subfactor planar algebra $P$ of modulus $\delta >1$. 
Let us denote by $M^{}_k(Q)$ the $II_1$-factor of section \S 2 corresponding to $F_k(Q)$.

We first observe some simple properties of conditional expectation-like maps associated to a compatible pair of finite pre-von Neumann algebras in the following lemma whose proof is simple and omitted.

\begin{lemma}\label{conditype1}
 Let $(A,t_A)\subseteq (B,t_B)$ be a compatible pair of finite pre-von Neumann algebras.
 Given an element $b\in B$ suppose there is  an element $a \in A$, 
 such that for any $c\in A$ the following holds true:
 \begin{equation}\label{trace}
  t_A(ac)= t_B(bc).
 \end{equation}
 Then, $a$ is necessarily unique, and denoting it by $E_A(b)$, the following relations hold:
 \begin{enumerate}
  \item $E_A(a)=a$ for all $a\in A$;
  \item $E_A(b^*)={E_A(b)}^*$; 
  \item $E_A(a_1ba_2)=a_1E_A(b)a_2$ for $a_1,a_2\in A$ and $b\in B$.\qed
 \end{enumerate}
\end{lemma}
%
We prove next that the compatible pair of finite pre-von Neumann algebras $(F_1(Q),t_1)\subseteq (F_1(P),t_1)$ always admits a conditional expectation-like map.
\begin{theorem}
 \label{conditype2}
For the compatible pair of finite pre-von Neumann algebras $$(F_1(Q),t_1)\subseteq (F_1(P),t_1),$$
there exists a $*$-preserving and $t_1$-preserving  $F_1(Q)-F_1(Q)$-bimodule retraction,
say $E$, for the usual inclusion of $F_1(Q)$ into $F_1(P).$
\end{theorem}
\begin{proof}
First observe that there exists a conditional expectation $E_{Q_n}$ from $P_n$ onto $Q_n$ such that
$$\tau(E_{Q_n}(x_n)y_n)=\tau(x_ny_n)~~~ \text{for all}~~~ x_n\in P_n, y_n \in Q_n.$$

Define $E: F_1(P) \rightarrow F_1(Q)$ as follows. 
For $x=(x_1,x_2,\cdots) \in F_1(P)$, set  $E(x)= (E_{Q_1}(x_1),E_{Q_2}(x_2),\cdots) \in F_1(Q).$
Next we claim that the following equation holds true
for all $(y_1,y_2,\cdots)\in F_1(Q)$.
\begin{equation}\label{a}
t_1((E_{Q_1}(x_1),E_{Q_2}(x_2),\cdots)\# (y_1,y_2,\cdots))=t_1((x_1,x_2,\cdots)\# (y_1,y_2,\cdots)).
\end{equation}

Here, by definition of the trace $t_1$, the left hand side is the trace of the $Q_1$ component of $(E_{Q_1}(x_1),E_{Q_2}(x_2),\cdots)\# (y_1,y_2,\cdots)$ while the right hand side is the trace of the $P_1$ component of $(x_1,x_2,\cdots)\# (y_1,y_2,\cdots).$
A little computation using the definition of the product $\#$ shows that Equation \ref{a} will follow once the  equation in Figure \ref{fig:treq} holds for
all $x_n \in P_n$ and $y_n \in Q_n$.
\begin{figure}[!h]
\begin{center}
\psfrag{x}{\tiny $x_n$}
\psfrag{y}{\tiny $y_n$}
\psfrag{eqx}{\tiny $E_{Q_n}(x_n)$}
\psfrag{2km2}{\tiny $2n-2$}
\includegraphics[scale=0.6]{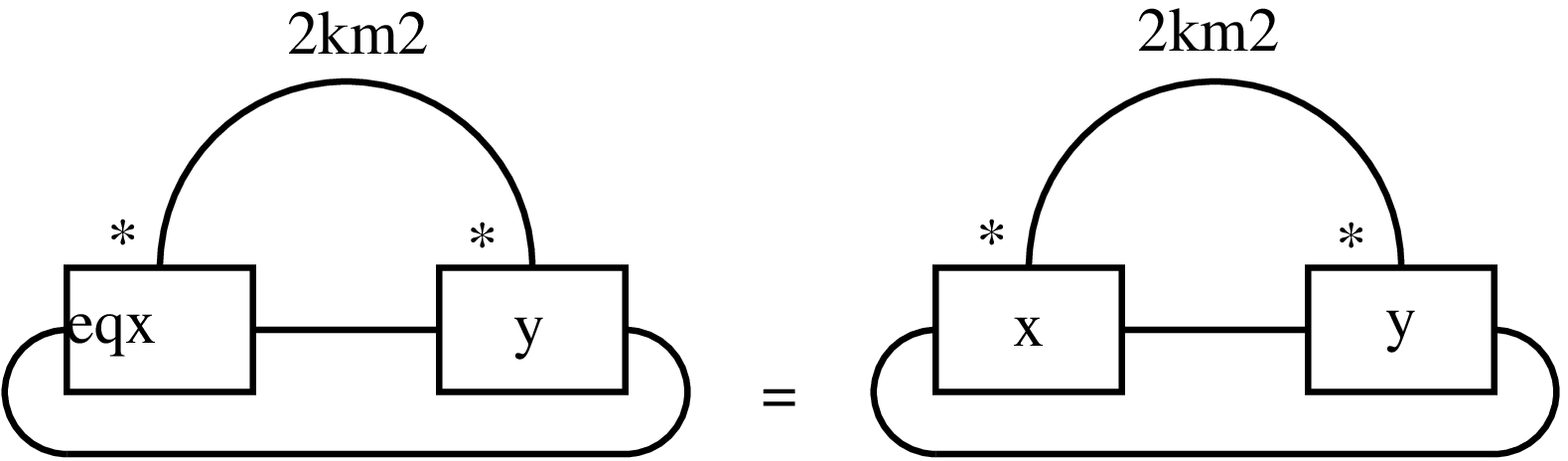}
\caption{}
\label{fig:treq}
\end{center}
\end{figure}

The left and right hand sides of this figure represent the traces of $E_{Q_n}(x_n)$ and $x_n$ against $Z_{R^{n-1}}(y_n)$ respectively. However since $Q$ is a planar subalgebra of $P$,
$Z_{R^{n-1}}(y_n) \in Q_n$, and by definition of the conditional expectation $E_{Q_n}$, the desired equality holds.

Finally, we appeal to Lemma \ref{conditype1} to complete the proof.
\end{proof}
In the following theorem we provide an example of a commuting square of finite pre-von Neumann algebras arising from a $*$-planar subalgebra.
\begin{theorem}\label{f}
Let $Q$ be a $*$-planar subalgebra of a subfactor planar algebra $P$ of modules $\delta > 1$.
 The following quadruple, call it $\mathscr{F}$, is a commuting square of finite pre-von Neumann algebras:
 $$\begin{matrix}
F_0(P) &\subseteq & F_1(P)\cr
\rotatebox{90}{$\subseteq$} &\ &\rotatebox{90}{$\subseteq$}\cr
F_0(Q)&\subseteq & F_1(Q).
\end{matrix}$$
\end{theorem}
\begin{proof}
 That each inclusion of finite pre-von Neumann algebras in the quadruple $\mathscr{F}$ is a compatible pair is obvious. Also, we have a
 $*$- and $\text{trace}$-preserving $F_0(P)-F_0(P)$  bimodule map $E_0 : F_1(P)\rightarrow F_0(P)$ which is  a retraction for the inclusion
 of $F_0(P)$ into $F_1(P).$ Moreover, by Theorem \ref{conditype2} there exists a $*$- and $t_1$- preserving
 bimodule map $E = E^{F_1(P)}_{F_1(Q)}$ which is also a retraction for the inclusion of $F_1(Q)$ into $F_1(P)$.
 As $F_0(P)\cap F_1(Q)= F_0(Q)$,  to show that $\mathscr{F}$ is a commuting square of finite pre-von Neumann algebras,  it suffices to show that the following equation holds:\
 \begin{equation*}\label{b}
  (E \circ E_0) (x_n)= (E_0\circ E) (x_n)~~~~~\forall~x_n \in P_n\subseteq F_1(P).
 \end{equation*}
 Computing with the definitions of $E$ and $E_0$, we see that it suffices to verify the pictorial equation on the left of Figure \ref{fig:condexpeq} for all $x_n \in P_n$, or equivalently, that for all $y_n \in Q_n$, the two elements of $P_n$ on the right of 
 Figure \ref{fig:condexpeq} have the same trace.
 \begin{figure}[!h]
\begin{center}
\psfrag{xn}{\tiny $x_n$}
\psfrag{yn}{\tiny $y_n$}
\psfrag{E}{\tiny $E_{Q_n}$}
\psfrag{eqnxn}{\tiny $E_{Q_n}(x_n)$}
\psfrag{2nm2}{\tiny $2n-2$}
\psfrag{nm2}{\tiny $n-2$}
\includegraphics[scale=0.4]{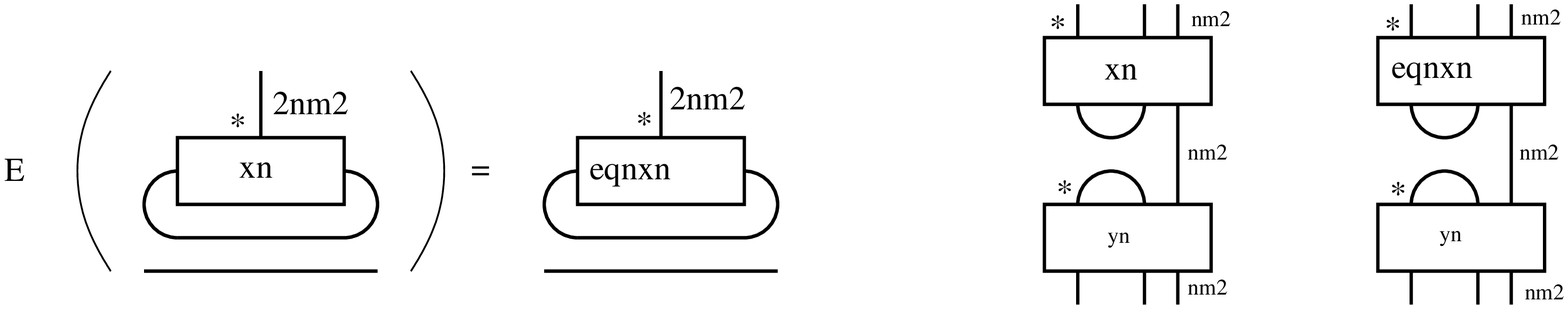}
\caption{}
\label{fig:condexpeq}
\end{center}
\end{figure}

Finally, this equality of traces holds since $Q$ is a planar subalgebra of $P$, just as in the proof of Theorem \ref{conditype2}.
\end{proof}

Before we state and prove the main result of this section, we need a lemma which
also follows from Theorem 7.1 of \cite{SnoWtn1994}. For completeness we sketch a simple proof.

\begin{lemma}\label{sym}
 Consider a commuting square ${\mathcal C}$ of type $II_1$ factors: 
$$
\begin{matrix}
L &\subseteq & M\cr
\rotatebox{90}{$\subseteq$} &\ &\rotatebox{90}{$\subseteq$}\cr
N &\subseteq & K
\end{matrix}
$$ 
with $[K:N]=[M:L]$. Then it is a nondegenerate commuting square, i.e., $\overline{\text{span}\{KL\}}= M = \overline{\text{span}\{LK\}}.$
\end{lemma}

\begin{proof}
Suppose that $\Lambda:= \{\lambda_i:i\in I=\{1,2,\cdots,n\}\}$ is a Pimsner-Popa basis for $K/N$. Thus, the matrix 
$q(K,N,\Lambda):=((q_{ij}))$, where $ q_{ij}= E^{K}_{N}(\lambda_i{\lambda}^*_j) ~~~\forall~~~ i,j
$, is a projection in $M_{n}(N)$ such that $tr(q(K,N,\Lambda)) = \frac{[K:N]}{n}$. Since by assumption $\mathcal C $ is a commuting square, we see that $E^{M}_{L}(\lambda_i{\lambda}^*_j)= E^{M}_{L} E^{M}_{K}(\lambda_i{\lambda}^*_j)= E^{M}_{N}(\lambda_i{\lambda}^*_j)= q_{ij}.$
Therefore, $q(M,L,\Lambda) = q(K,N,\Lambda)$ and further $tr (q(M,L,\Lambda)) = tr(q(K,N,\Lambda))= \frac{[K:N]}{n}= \frac{[M:L]}{n}$. This proves that $\{\lambda_i:i\in I\}$ is also 
a basis for $M/L.$ Now the non-degeneracy of ${\mathcal C}$ follows from \cite{Ppa1994}. This completes the proof.
\end{proof}
We are now ready to deduce the main result of this section.
\begin{theorem}
 Suppose $Q$ is a $*$-planar subalgebra of the subfactor planar algebra $P$  (of modulus $\delta>1)$. Then there exists a smooth non-degenerate commuting square of type $II_1$-factors:
 \[\begin{array}{ccc}
  M^{}_0(P) & \subseteq &  M^{}_1(P)\\
\rotatebox{90}{$\subseteq$} & & \rotatebox{90}{$\subseteq$}\\
  M^{}_0(Q) & \subseteq & M^{}_1(Q)
\end{array}\]  such that planar algebra of $M^{}_0(P) \subseteq  M^{}_1(P)$ is isomorphic to $P$ and 
the planar algebra of $M^{}_0(Q) \subseteq M^{}_1(Q)$ is isomorphic to $Q$.
\end{theorem}

\begin{proof}
 By Theorem \ref{f} it follows that the quadruple $\mathscr{F}$ defined as follows
 $$\begin{matrix}
F_0(P) &\subseteq & F_1(P)\cr
\rotatebox{90}{$\subseteq$} &\ &\rotatebox{90}{$\subseteq$}\cr
F_0(Q)&\subseteq & F_1(Q)
\end{matrix}$$ is a commuting square of finite pre-von Neumann algebras.
 Next, apply Theorem \ref{lambda} to obtain a commuting square $\mathscr{G}$ of type $II_1$ factors as follows:
 \[\begin{array}{ccc}
  M^{}_0(P) & \subseteq &  M^{}_1(P)\\
\rotatebox{90}{$\subseteq$} & & \rotatebox{90}{$\subseteq$}\\
  M^{}_0(Q) & \subseteq & M^{}_1(Q).
\end{array}\]

\noindent But by Theorem \ref{gjsks} we know that the planar algebra of the extremal subfactor $ M^{}_0(P)  \subseteq M^{}_1(P)$ is isomorphic to $P$ and is of index
${\delta}^2$. Similarly,  the planar algebra of $M^{}_0(Q) \subseteq M^{}_1(Q)$ is isomorphic to $Q$ and is also of index ${\delta}^2$ (since $Q$ is a planar subalgebra of $P$).
Therefore, $[M^{}_1(P):M^{}_0(P)]= [M^{}_1(Q):M^{}_0(Q)]= {\delta}^2$.
Then we apply Lemma \ref{sym} to conclude that $\mathscr{G}$ is a non-degenerate commuting square.
Finally recall that (see for example Proposition 5.2 in \cite{KdySnd2009}) 
$\big(M^{}_0(Q)\big)^{\prime}\cap M^{}_k(Q)= Q_k \subseteq P_k = \big(M^{}_0(P)\big)^{\prime}\cap M^{}_k(P) $. Thus $\mathscr{G}$ is a smooth non-degenerate commuting square.
This completes the proof.
\end{proof}

\end{document}